\documentclass[10pt,reqno]{article}
\usepackage[top=1in, bottom=1in, left=1in, right=1in]{geometry}
\usepackage{graphicx}
\usepackage[dvipsnames]{xcolor}
\usepackage{amssymb, amsthm, mathrsfs, amsmath, amsfonts,dsfont}
\usepackage{epstopdf}
\usepackage{enumerate, enumitem,fancyhdr}
\usepackage{lipsum}

\usepackage{hyperref}

 \usepackage{authblk}
\usepackage{blindtext}

\usepackage{lipsum}

\newcommand\blfootnote[1]{%
  \begingroup
  \renewcommand\thefootnote{}\footnote{#1}%
  \addtocounter{footnote}{-1}%
  \endgroup
}

\newtheorem{thm}{Theorem}[section]
\newtheorem{lem}{Lemma}[section]
\newtheorem{prop}{Proposition}[section]
\newtheorem{cor}{Corollary}[section]

\theoremstyle{definition}

\newtheorem{rem}{Remark}[section]

\setenumerate[0]{label=(\roman*)}

\def\R{\mathbb{R}}
\def\N{\mathbb{N}}
\def\Z{\mathbb{Z}}
\def\T{\mathbb{T}}
\def\C{\mathbb{C}}
\def\W8{W_{\infty}}
\def\d{\delta}
\def\h{h'}
\def\D{\Delta}
\def\d*{\delta_*}
\def\Em{\mathbb{E}_{\mu_\gamma}}
\def\En{\mathbb{E}_{\nu^r_\gamma}}
\def\w{\omega}
\def\div{{\rm div}}
\def\E{ E}

\begin{document}
\date{}

\title{Invariant measures for the two-dimensional averaged-Euler equations}
\author{Alexandra Symeonides}

\affil{GFMUL and Dep. de Matem\'atica da Faculdade de Ci\^encias, Univ. de Lisboa, \\
Campo Grande, 1749-016 Lisboa, Portugal.}

\maketitle

\begin{abstract}\noindent
We define a Gaussian invariant measure for the two-dimensional averaged-Euler equation and show the existence of its solution 
with initial conditions on the support of the measure. An invariant surface measure on the level sets of the energy is also constructed, as well
as the corresponding flow. Poincar\'{e} recurrence theorem is used to show that the flow returns infinitely many times in a neighborhood of the initial state.
\end{abstract}

\tableofcontents

\section{Introduction}
The purpose of this paper is to built invariant measures for the averaged-Euler equations. The averaged-Euler equations were introduced in 1998 by D. D. Holm, J. E. Marsden and T. S. Ratiu in \cite{DMR}. For an incompressible non-viscous fluid the equations are the following
\begin{equation*}
\frac{\partial Au}{\partial t} + (u \cdot \nabla)Au +(\nabla u)^T\cdot Au=-\nabla p, \qquad \nabla\cdot u=0,
\end{equation*}
where $A=(1-a^2\D)^s$ for $a$ a real parameter and $s$ a positive number. The mean velocity of the flow is denoted by $u:\R^2\to\R^2$ and the pressure by $p:\R^2\to\R$. The authors of \cite{DMR} consider a modification of the Euler equations such that non linear effects at small scales of the motion are negligible; therefore the dynamics remains turbulent, but non dissipative. 

Global existence and uniqueness for solutions of the two-dimensional averaged-Euler equation are known both in $\R^2$ and in a bounded domain for initial velocities in $H^3$, see respectively V. Busuioc \cite{Bus} and S. Shkoller \cite{Shk}. In the latter reference  the classical pde problem is transformed into a geometric one, considering geodesics in the (infinite-dimensional) group of volume-preserving diffeomorphisms. Averaged-Euler equations are indeed known to describe the velocities of geodesics on this group endowed with the $H^1$ metric. For further results concerning these equations we cite \cite{BT} and references therein.

For this system O. Bell, A. Chorin and W. Crutchfield pointed out in \cite{BCC} that invariant Gibbs measures can be considered, as the equations conserve the energy and the enstrophy. In their perspective the invariant measures are used to perform numerical predictions of the dynamics. Invariant measures are of significant interest when employed to improve existing deterministic results, particularly  when we deal with vector fields with low regularity, see for example \cite{C3}, \cite{AF}, \cite{FL}; and may be used, among others, to extend local to global existence results or to prove recurrence properties for a flow, see for example \cite{TzvVis}. Also quasiinvariant measures may serve to the same purposes and besides they are sometimes supported on more regular spaces, see for example \cite{Tzv}. Finally we mention several works related to invariant measures: \cite{AC}, \cite{AHKD}, \cite{C}, \cite{AFerr} about the two-dimensional Euler equations; and \cite{Bourg}, \cite{CSuzz}, \cite{ASSuzz} about other dispersive equations.

We consider an equivalent formulation, the vorticity formulation, of the averaged-Euler equations. On $\R^2$, for divergence free velocity fields, the ``stream function'' $\varphi: \R^2 \to \R$ is defined by $u=\nabla^\perp\varphi:=(-\partial_2 \varphi, \partial_1\varphi)$. The vorticity formulation in terms of stream function is the following
\begin{equation}\label{equation}
\frac{\partial A\D\varphi}{\partial t} + (\nabla^\perp\varphi\cdot \nabla)A\D\varphi=0.
\end{equation}
Below we consider \eqref{equation} on $\T \simeq [0,2\pi]\times[0,2\pi]$ and with periodic boundary conditions.

In this paper we rigorously define for this system an infinite dimensional Gibbs measure with respect to the enstrophy, formally
$$
``d\mu_\gamma\simeq \frac{1}{Z} e^{-\gamma \frac{Enstrophy}{2}}d\lambda",
$$
where $d\lambda$ denotes ``Lebesgue measure", $Z$ a normalizing constant and $\gamma$ is a positive parameter. We construct a flow for the averaged-Euler equations on the support of this measure.
Namely, as previously done by A. B. Cruzeiro and S. Albeverio in \cite{AC} for the analogous case of the Euler equations, applying a combination of Prohorov and Skorohod's  theorems to  finite dimensional flow approximations, we can construct continuous flows for the averaged-Euler vector field on some probability space $(\Omega, \mathcal{F}, P_\gamma)$ with values in $H^{1-\alpha,s}$ for some $s>0$ and $\alpha >2$, that is in a  Sobolev space of negative order. Therefore these pointwise continuous flows belong to a distribution space. In particular we will have
$$
U(t,\w)=U(0,\w)+\int_0^t B(U(s,\w)) ds, \qquad P_\gamma-a.e. ~\w \in \Omega, ~ \forall t\in\R,
$$
with $\mu_\gamma$ invariant under the flow. See Theorem \ref{exists} below.

We also consider, as previously done for the Euler equations by F. Cipriano in \cite{C}, the infinite-dimensional conditional measure defined on level sets of the energy. Comparing with the Euler case, there is no need here to define a renormalized energy, since in the averaged-Euler case the energy itself is square integrable with respect to $\mu_\gamma$. This surface measure  $\nu_\gamma^r$, is also invariant and therefore pointwise continuous flows can be constructed on a probability space with values on the level sets of the energy, say $V_r$, for some positive $r$. We have
$$
U'(t,\w)=U'(0,\w) + \int_0^t B^*(U'(s,\w))ds, \qquad P^r_\gamma-a.e.~\w, ~ \forall t\in\R,
$$
where $B^*$ is any redefinition of $B$. Moreover $\nu^r_\gamma$ is invariant under the flow. See Theorem \ref{flow_level_sets} below.

Finally since the Poincar\'{e} recurrence theorem holds, we have that the flow returns to a neighborhood of the initial state infinitely many times. The analogous for the Euler system was proved in \cite{CL} by A. Constantin and D. Levy.

In Section 2 we define the spaces of functions which are relevant in our work; we rewrite the vorticity formulation for the averaged-Euler equations as an infinite dimensional system of ordinary differential equations using the Fourier coefficients of the stream function. Here we also show that the energy and the enstrophy are conserved quantities. In Section 3 we rigorously define a Gibbs measure $\mu_\gamma$  and describe its support. Moreover we study the $L^p_{\mu_\gamma}$ regularity of the vector field and we show that it is divergence free with respect to $\mu_\gamma$. Finally we construct a flow on a suitable probability space. In Section 4 we define the infinite dimensional conditional measure $\nu_\gamma^r$ with support on the level sets of the energy $V_r$. In Section 5 we show that $\nu_\gamma^r$ is invariant and prove existence of a flow defined on some probability space. Finally we show that the solution returns to a vicinity of the initial state infinitely many times.

\section{The averaged-Euler equations}\label{The averaged-Euler equations}

Consider the operator $A=(1-a^2\D)^s$ for $a$ a real parameter and $s$ a positive number; if $s$ is not an integer, $A$ is a pseudo-differential operator. 
The averaged-Euler equations for an incompressible non-viscous fluid on $\R^2$ are (c.f. \cite{BCC}, \cite{DMR})
\begin{equation}\label{avEuler}
\frac{\partial Au}{\partial t} + (u \cdot \nabla)Au +(\nabla u)^T\cdot Au=-\nabla p, \qquad \nabla\cdot u=0
\end{equation}
where $u:\R^2\to\R^2$ is the velocity of the flow and $p:\R^2\to\R$ is the pressure. 
In what follows we denote $\nabla^\perp=(-\partial_2,\partial_1)$ where $\partial_1,\partial_2$ are the partial derivatives with respect to the first and second variable.

We have the following
\begin{thm}
A time dependent vector field $u$ is a smooth solution of \eqref{avEuler} if and only if there exists a smooth (real) function $\varphi$ such that $u=\nabla^\perp \varphi$ and $\varphi$ is a solution of the equation
\begin{equation}\label{vorAvEuler}
\frac{\partial A\D\varphi}{\partial t} + (\nabla^\perp\varphi\cdot \nabla)A\D\varphi=0.
\end{equation}
\end{thm}
\begin{proof} 
Taking the ``curl'' of \eqref{avEuler},
$$
\frac{\partial A \nabla^\perp \cdot u}{\partial t} + \nabla^\perp \cdot [(u \cdot \nabla)Au] + \nabla^\perp \cdot [(\nabla u)^T\cdot Au]=0,
$$
we get
$$
\frac{\partial A\nabla^\perp \cdot u}{\partial t} + (u\cdot \nabla) A\nabla^\perp\cdot u=0.
$$ 
From the condition $\nabla\cdot u=0$ we know that exists a real-valued function $\varphi$, called the stream function, such that $u=\nabla^\perp\varphi$; thus sufficiency is proved. To prove necessity let $f$ be defined by
$$
f=-\frac{\partial A \nabla^\perp \varphi}{\partial t} - (\nabla^\perp\varphi \cdot \nabla)A\nabla^\perp\varphi - (\nabla\nabla^\perp\varphi)^T\cdot A\nabla^\perp\varphi.
$$
Taking the ``curl'' we get
$$
-\frac{\partial A\D\varphi}{\partial t} - (\nabla^\perp\varphi\cdot \nabla)A\D\varphi=\nabla^\perp \cdot f,
$$
then $\nabla^\perp \cdot f=0$ by assumption and thus there exists a scalar function $p$ such that $f=\nabla p$.
The proof is performed in detail in \cite{AHKD} for the analogous case of the Euler system.
\end{proof}
Our general settings will be similar to the ones in \cite{C,AC,AHKD} where the case of Euler equation is studied.
We will consider our equations on the two-dimensional torus $\T \simeq [0,2\pi]\times[0,2\pi]$ and with periodic boundary conditions, that is 
$$
\varphi(0,y,t)=\varphi(2\pi,y,t) \mbox{ and } \varphi(x,0,t)=\varphi(x,2\pi,t), \qquad \forall (x,y)\in\T.
$$

\begin{rem}
From the expression of the vorticity equation we remark that if $s=0$ we are considering the Euler system.
\end{rem}

\subsection{Conserved quantities of the motion}

The averaged-Euler equation is conservative, meaning that the ``energy'' $(u,Au)$ is an invariant of the motion (the inner product is the one of $L^2(\T)$); also the ``enstrophy'' $(A\nabla^\perp \cdot u,A\nabla^\perp \cdot u)$ is a conserved quantity and we can write these quantities  in terms of the stream function $\varphi$ as
$$
E=-\frac{1}{2}\int_{\T} \varphi A\D \varphi dx
$$
and
$$
S=\frac{1}{2}\int_{\T} (A\D\varphi)^2 dx.
$$
We have in fact that
\begin{align*}
\frac{d E}{dt}&
=-\left( \frac{\partial }{\partial t} A\D \varphi , \varphi \right)=\left((\nabla^\perp\varphi\cdot \nabla)A\D\varphi, \varphi\right)\\
&=\left(\nabla^\perp\varphi\cdot \nabla \varphi,A\D\varphi\right)=0
\end{align*}
and 
\begin{align*}
\frac{d }{dt}S&=\left( \frac{\partial A\D \varphi}{\partial t},A\D \varphi\right)=-\left((\nabla^\perp\varphi\cdot \nabla)A\D\varphi, A\D \varphi\right)\\
&=-\left(\nabla^\perp\varphi\cdot \nabla \varphi, A\D A\D \varphi\right)=0
\end{align*}
since $\nabla^\perp\varphi\cdot \nabla \varphi=0$. 

\subsection{Fourier expansion of the system}
 We want to write the averaged-Euler partial differential equation as an infinite dimensional ordinary differential equation by means of Fourier expansion series
 (c.f. \cite{AC,AHKD} for an analogous formulation of the Euler equation). We consider an orthonormal basis of $L^2(\T)$, $\{e_k(x)\}_{k\in \Z^2}$, defined by
$e_k(x)=\frac{1}{2\pi}e^{ik\cdot x}$. These functions are eigenfunctions of the Laplace operator:
$$
\D e_k(x)=-k^2 e_k(x), \qquad \forall k\in \Z^2.
$$
Here $k\cdot x=k_1x_1+ k_2x_2$ for $k=(k_1,k_2)\in\Z^2$ and $x=(x_1,x_2)\in \T$ and $k^2=k\cdot k$. We say that $k\in\Z^2$ is positive if $k_1>0$ or $k_1=0$ and $k_2>0$. The Sobolev spaces 
$$\mathcal{H}^{2s+2}(\T)=\left\{ v: \T\to \R : \int \sum_{|\alpha|\leq 2s+2} |D^\alpha v(x)|^2 dx<+\infty \right\}$$ can be identified with the complex Hilbert spaces 
$$
H^{2,s}=\left\{v=\sum_{k\in\Z^2}v_ke_k : \sum_{k>0} k^4(1+a^2k^2)^{2s}|v_k|^2<+\infty \right\}
$$
with inner product $<u,v>_{2,s}=\sum_{k>0}  k^4(1+a^2k^2)^{2s}u_k\bar v_k$. For general $p\in\R$ we define
$$
H^{p,s}=\left\{v=\sum_{k\in\Z^2}v_ke_k : \sum_{k>0} k^{2p}(1+a^2k^2)^{ps}|v_k|^2<+\infty \right\}
$$
with inner product $<u,v>_{p,s}=\sum_{k>0}  k^{2p}(1+a^2k^2)^{ps}u_k\bar v_k$. 

Henceforth we write $\varphi(x,t)=\sum_{h>0}\omega_h(t)e_h(x)$ and we write the energy and the enstrophy as
$$
E=\frac{1}{2}\|\varphi\|_{1,s}^2
$$
and
$$
S=\frac{1}{2}\|\varphi\|_{2,s}^2.
$$
Set $\h^\perp=(-\h_2, \h_1)$ for $\h\in\Z^2$; we have
\begin{align*}
\nabla^\perp\varphi\cdot \nabla A\D\varphi&=-\frac{1}{2\pi}\sum_{\substack{h>0, \\ \h>0, \\ \h\neq h}} \w_h \w_{\h}
(h\cdot \h^\perp) \h^2 (1+a^2\h^2)^s e_{h+\h}(x)\\
&=-\frac{1}{2\pi}\sum_{\substack{h>0, \\ k>0, \\ \h+h=k}} \w_h \w_{\h}
(h\cdot \h^\perp) \h^2 (1+a^2\h^2)^s e_k(x).
\end{align*}
Hence equation  \eqref{vorAvEuler} holds if and only if 
$$
-\sum_{k>0}\left[ k^2(1+a^2k^2)^s \frac{d \w_k}{dt} + \frac{1}{2\pi}\sum_{\substack{h+\h=k, \\ h>0}}  \w_h \w_{\h}
(h\cdot \h^\perp) \h^2 (1+a^2\h^2)^s  \right]e_k(x)=0,
$$
meaning that \eqref{vorAvEuler} can be written as an infinite dimensional ODE as follows:
\begin{align}\label{vorAvEulerFourier}
2\pi k^2(1+a^2k^2)^s \frac{d \w_k}{dt}&= -\sum_{\substack{h+\h=k, \\ h>0}}  (h\cdot \h^\perp) \h^2 (1+a^2\h^2)^s \w_h \w_{\h} \nonumber \\
&=\frac{1}{2}\sum_{\substack{h+\h=k, \\ h>0}} (h\cdot \h^\perp) [h^2 - \h^2] (1+a^2\h^2)^s \w_h \w_{\h},\quad \forall k>0.
\end{align}
From $\h^\perp\cdot h=-\h\cdot h^\perp$ we obtain the following form of the averaged-Euler equations:
$$
\frac{d \w_k}{dt}=B_k(\varphi), \qquad \forall k>0,
$$
where the vector field $B$ is defined by
\begin{equation}\label{vector_fieldB}
B(\varphi)=\sum_{k} B_k(\varphi)e_k,
\end{equation}
with
\begin{equation}\label{vector_field}
B_k(\varphi)=\frac{1}{2\pi }\sum_{h>0} \left[\frac{1}{k^2 }(h^\perp \cdot k)(h\cdot k)- \frac{1}{2}(h^\perp\cdot k)\right]\frac{(1+ a^2 (k-h)^2)^s}{(1+a^2 k^2)^s} \w_h \w_{k-h}.
\end{equation}

\section{Gaussian invariant measures for the averaged-Euler equations}

The purpose of this section is to construct an infinite dimensional Wiener measure $\mu_\gamma$ defined  on some suitable $H^{p,s}$ space which is invariant for the averaged-Euler equation. We   study the $L^p_{\mu_\gamma}$ regularity of the averaged-Euler vector field $B$, which  is $\mu_\gamma$-divergence free. This latter property will unable us to construct a flow associated with $B$. We proceed as in \cite{AC} where the Euler equation is considered.

Consider the probability measures on $\C$ defined for $\gamma\in \R^+$ by
\begin{equation}
d\mu_{\gamma,k}(z)=\frac{\gamma k^4(1+a^2k^2)^{2s}}{2\pi}\exp\left\{-\frac{1}{2}\gamma k^4(1+a^2 k^2)^{2s}|z|^2\right\}dx dy
\end{equation}
where $z=x+iy$. Then
\begin{equation}\label{enstrophy_measure}
d\mu_\gamma(\varphi)=\prod_{k>0}d\mu_{\gamma,k}(\w_k)
\end{equation}
is a measure with support in $H^{1-\alpha,s}$ for any $\alpha>-\frac{s}{s+1}$; indeed,
\begin{align*}
\int \|\varphi\|^2_{1-\alpha,s} d\mu_\gamma(\varphi)&=\sum_{k>0}k^{2(1-\alpha)}(1+a^2k^2)^{s(1-\alpha)} \prod_{h>0}\int |\w_k|^2 d\mu_{\gamma,h}(\w_h) \\ 
&= \frac{2}{\gamma}\sum_{k>0} \frac{1}{k^{2(1+\alpha)}(1+a^2k^2)^{s(1+\alpha)}} <+\infty, \qquad \forall \alpha>-\frac{s}{s+1}.
\end{align*}

\begin{prop}\label{ab_W_space}
$(H^{1-\alpha,s}, H^{2,s}, \mu_\gamma)$ is a complex abstract Wiener space with measurable norm $\| \cdot\|_{1-\alpha,s}$ for any $\alpha>-\frac{s}{s+1}$.
\end{prop}
\begin{proof}
Consider the operator $\Gamma: H^{2,s}\to H^{2,s}$ defined by 
$$\Gamma e_k= \frac{1}{|k|^{1+\alpha}(1+a^2 k^2)^{(1+\alpha)s/2}} e_k,$$
which is a Hilbert-Schmidt operator  since
$$\|\Gamma\|_{H.S.}^2=\sum_{k>0} \frac{1}{k^{2(1+\alpha)}(1+a^2 k^2)^{(1+\alpha)s}}<+\infty, \qquad \forall \alpha>-\frac{s}{s+1}.
$$
Let now $u=\sum_k u_k e_k$ in $H^{2,s}$, then
$$
\|\Gamma u\|_{2,s}^2 = \sum_k k^{2(1-\alpha)}(1+a^2k^2)^{s(1-\alpha)} |u_k|^2=\|u\|^2_{1-\alpha,s}.
$$
Because  $\Gamma$ is a Hilbert-Schimdt operator such that $\|\Gamma u\|_{2,s}^2 = \|u\|^2_{1-\alpha,s}$ we can say that $\| \cdot\|_{1-\alpha,s}$ is a measurable norm in the sense of Gross (that is for every $\varepsilon>0$ there exists $P_0\in\mathcal{F}$, where $\mathcal{F}$ is the partially ordered set of finite dimensional orthogonal projection $P$ of the space $H^{2,s}$, such that $\mu_\gamma\{\|P u\|_{1-\alpha,s}>\varepsilon\}<\varepsilon,\, \forall P\perp P_0\in\mathcal{F}$). On the other hand $H^{1-\alpha,s}$ is the closure of $H^{2,s}$ with respect to the norm $\| \cdot\|_{1-\alpha,s}$, that is $(H^{1-\alpha,s}, H^{2,s}, \mu_\gamma)$ is a complex abstract Wiener space. Indeed $\mu_\gamma$ is a Wiener measure on $H^{1-\alpha,s}$, namely
$$
\int e^{i\gamma l(u)} d\mu_\gamma(u)=e^{-\frac{1}{2}\gamma\| l\|_{2,s}^2}, \qquad \forall l\in (H^{1-\alpha,s})'\subset H^{2,s}.
$$
\end{proof}
In particular $\Em(u_k)=0$, $\Em(u_k \bar u_{k'})=\frac{2\delta_{k,k'}}{\gamma k^4(1+a^2k^2)^{2s}}$ and $\Em(|u_k|^{2p})=\frac{2^p p!}{\gamma^p k^{4p}(1+a^2 k^2)^{2sp}}$ for $p\geq 1$. For further studies on abstract Wiener spaces and results similar to Proposition \ref{ab_W_space} see \cite{K}.

\subsection{Regularity of the vector field}\label{reg}

We are looking for solutions of
$$
\frac{d \w_k}{dt}=B_k(\varphi), \qquad \forall k>0
$$
that belong to $H^{1-\alpha,s}$ for all $t>0$. Let us prove that  $B: H^{1-\alpha,s} \to H^{1-\alpha,s}$, where $B$ is defined in \eqref{vector_fieldB}.
We can consider  finite dimensional approximations of $B$, namely
\begin{equation}\label{B_n}
B^n(\varphi)=\sum_{k\in \{\alpha_1, \ldots, \alpha_{d(n)}\}} B_k^n(\varphi) e_k(x),
\end{equation}
that are vector fields on $\C^d$, $d=d(n)$, and write 
\begin{equation}\label{B_n^k}
B_k^n(\varphi)=\sum_{h^2\leq n} \alpha_{h,k} \w_h \w_{k-h}
\end{equation}
where $\alpha_i \in \Z^2$ for $i\in\{1, \ldots, d\}$
and
$$
\alpha_{h,k}=\frac{1}{2\pi } \left[\frac{1}{k^2 }(h^\perp \cdot k)(h\cdot k)- \frac{1}{2}(h^\perp\cdot k)\right]\frac{(1+ a^2 (k-h)^2)^s}{(1+a^2 k^2)^s}.
$$

\begin{prop}\label{B_regularity}
The vector field $B\in L^p_{\mu_\gamma}(H^{1-\alpha,s};H^{1-\alpha,s})$ for all $\alpha>2$ and $p\geq 1$.
\end{prop}
\begin{proof}
It is sufficient to prove that $B\in L^{2p}_{\mu_\gamma}(H^{1-\alpha,s};H^{1-\alpha,s})$ for all $\alpha>2$ with $p$ odd.
\begin{align*}
\Em&\|B(\varphi)\|^{2p}_{1-\alpha,s}=\Em\left( \sum_k k^{2(1-\alpha)}(1+a^2k^2)^{(1-\alpha)s}|B_k(\varphi)|^2\right)^p \\
&=\Em \sum_{k_1,..., k_p} \prod_{i=1}^p k_i^{2(1-\alpha)}(1+a^2k_i^2)^{(1-\alpha)s} |B_{k_i}(\varphi)|^2 \\
&\leq \sum_{k_1,..., k_p} \prod_{i=1}^p k_i^{2(1-\alpha)}(1+a^2k_i^2)^{(1-\alpha)s} \left(\Em |B_{k_i}(\varphi)|^{2p} \right)^{1/p} \\
&=\left[ \sum_{k}  k^{2(1-\alpha)}(1+a^2k^2)^{(1-\alpha)s} \left(\Em |B_{k}(\varphi)|^{2p} \right)^{1/p}\right]^p.
\end{align*} 
From $B_{k}(\varphi)=\sum_h \alpha_{h,k}\w_h \w_{k-h}$ we get that
\begin{align*}
\Em |B_k(\varphi)|^{2p}&=\Em\left( \sum_{\substack{h_1, ... h_p \\ \h_1, ... \h_p}} \prod_{i=1}^p \alpha_{h_i,k}\alpha_{\h_i,k} \w_{h_i}\w_{k-h_i} \bar\w_{\h_i}\bar\w_{k-\h_i} \right) \\
&\leq \left[ \sum_{h,\h} \alpha_{h,k}\alpha_{\h,k} (\Em(\w_{h}\w_{k-h} \bar\w_{\h}\bar\w_{k-\h} )^p)^{1/p}\right]^p.
\end{align*}
Observe that, if $h\neq\h$ or $h\neq k-\h$,
$$
\Em(\w_{h}\w_{k-h} \bar\w_{\h}\bar\w_{k-\h} )^p=\Em(\w_{h})^p\Em( \w_{k-h} \bar\w_{\h}\bar\w_{k-\h})^p
$$
and for $p$ odd $\Em(\w_{h})^p=0$. Hence we have  
\begin{align*}
 \sum_{h,\h\neq 0,k}  &\alpha_{h,k}\alpha_{\h,k} (\Em(\w_{h}\w_{k-h} \bar\w_{\h}\bar\w_{k-\h} )^p)^{1/p} \\
&=  \sum_{h,\h\neq 0,k} (\delta_{h,k-\h}\alpha_{h,k}\alpha_{\h,k} + \delta_{h,\h}\alpha_{h,k}\alpha_{\h,k})(\Em (|\w_h|^{2p}|\w_{k-h}|^{2p}))^{1/p} \\
& \leq 2\sum_{h\neq 0,k} |\alpha_{h,k}|^2(\Em |\w_h|^{2p})^{1/p}(\Em |\w_{k-h}|^{2p})^{1/p} \\
& =  2\sum_{h\neq 0,k}|\alpha_{h,k}|^2 \frac{(2^p p!)^{1/p}}{\gamma h^4(1+a^2h^2)^{2s}} \frac{(2^p p!)^{1/p}}{\gamma (k-h)^4(1+a^2(k-h)^2)^{2s}} \\
& \leq c(p,\gamma)\frac{1}{(1+a^2k^2)^{2s}} \sum_{h\neq 0,k} \left[ \frac{1}{4h^2(k-h)^2(1+a^2h^2)^{2s}}\right]
\end{align*}
and thus
\begin{align*}
\sum_{k}  &k^{2(1-\alpha)}(1+a^2k^2)^{(1-\alpha)s} \left(\Em |B_{k}(\varphi)|^{2p} \right)^{1/p} \\
&\leq c(p,\gamma) \sum_{k}  k^{2(1-\alpha)}(1+a^2k^2)^{-s(1+\alpha)}  \sum_{h\neq 0,k} \left[ \frac{1}{4h^2(k-h)^2(1+a^2h^2)^{2s}}\right]
\end{align*}
where the series converge for $\alpha>2$.
\end{proof}

\begin{cor}\label{convergence***}
The convergence $\lim_{n\to +\infty}B^n= B$ holds in $L^2_{\mu_\gamma}(H^{1-\alpha,s}; H^{1-\alpha,s})$. 
\end{cor}
\begin{proof}
To show this statement observe that
\begin{equation*}
\Em(\|B^n(\varphi)-B(\varphi)\|^2_{1-\alpha,s})=\sum_{k>0}  k^{2(1-\alpha)}(1+a^2k^2)^{(1-\alpha)s} \Em(|B^n_k(\varphi)-B_k(\varphi)|^2 )\leq \varepsilon
\end{equation*} 
for $\alpha>2$ and $n$ sufficiently big. In fact $\Em(|B^n_k(\varphi)-B_k(\varphi)|^2 )$ is infinitesimal for $n$ sufficiently big, since $\lim_{n\to +\infty}B^n_k(\varphi)= B_k(\varphi)$ for a.e $\varphi\in H^{1-\alpha,s}$ and $B^n_k$ is a Cauchy sequence in $L^2_{\mu_\gamma}(H^{1-\alpha,s}; \C)$, that is for $0<n<m$
\begin{align*}
\Em|B^n_k&(\varphi)-B^m_k(\varphi)|^2=\sum_{\substack{n\leq h^2\leq m \\ n\leq \h^2\leq m }} \alpha_{h,\h} \alpha_{\h,k} \Em(\w_h\w_{k-h}\bar\w_{\h} \bar\w_{k-\h})\\
&= \frac{4}{\gamma^2} \sum_{\substack{n\leq h^2\leq m \\ n\leq \h^2\leq m }} \alpha_{h,\h} \alpha_{\h,k} \frac{(\delta_{h,\h}  + \delta_{h, k-\h})}{h^4(1+a^2h^2)^{2s}(k-h)^4(1+a^2(k-h)^2)^{2s}}\\
&\leq \frac{8}{\gamma^2}\sum_{n\leq h^2\leq m}\frac{ |\alpha_{h,k}|^2 }{h^4(1+a^2h^2)^{2s}(k-h)^4(1+a^2(k-h)^2)^{2s}}<+\infty
\end{align*}
as we saw above.
\end{proof}

We shall consider the gradient operator in the sense of Malliavin calculus (c.f. \cite{Ma}), that is, for $\psi: H^{1-\alpha,s} \to X$ where $X$ is a Banach space, $\nabla \psi (u)$ is defined, for $u\in H^{1-\alpha,s}$, by 
$$
\nabla \psi (u)(v)=D_v\psi (u)=\lim_{\varepsilon \to 0} \frac{1}{\varepsilon}[\psi(u+\varepsilon v)-\psi(u)], \quad v\in  H^{2,s}
$$
where the limit is taken $\mu_\gamma$-a.e. in $H^{1-\alpha,s}$.
The second derivative is defined by iteration of the first, that is $\nabla^2 \psi (u)(v,w)=D_vD_w\psi(u)$ for $u\in H^{1-\alpha,s}$ and $v,w \in H^{2,s}$, etc. Observe that the successive gradients $\nabla^r \psi(u)$ belong to $H^{2,s}_{\otimes^r}$ for $r\geq 1$. On the symmetric tensorial product $H^{2,s}_{\otimes^r}=H^{2,s}\otimes\cdots\otimes H^{2,s}$ (r times) we consider the Hilbert-Schmidt norm. 

For example, we can check that
\begin{equation}\label{first_derivative_B}
D_{e_j}B(\varphi)=\sum_{k>0}(\alpha_{j,k}+ \alpha_{k-j,k})\w_{k-j}e_k
\end{equation}
and
\begin{equation}\label{second_derivative_B}
D_{e_i}D_{e_j}B(\varphi)=\sum_{\substack{k=i+j, \\ k>0}}(\alpha_{j,k}+ \alpha_{i,k})e_{k}.
\end{equation}
Hence given $\{\hat e_k\}_{k\in \Z^2}$ an orthonormal basis of $H^{2,s}$, namely $\hat e_k=\frac{e_k}{k^2(1+a^2k^2)^s}$, the Hilbert-Schmidt norms of $\nabla B$ and $\nabla^2 B$ are respectively
\begin{align}\label{HS_grad}
\|\nabla B(\varphi)\|^2_{H.S.}&=\sum_j\|\nabla B(\varphi)(\hat e_j)\|^2_{1-\alpha,s}  \nonumber\\
&= \sum_{j,k}\frac{k^{2(1-\alpha)}(1+a^2k^2)^{s(1-\alpha)}}{j^4(1+a^2j^2)^{2s}}
(\alpha_{j,k}+ \alpha_{k-j,k})^2|\w_{k-j}|^2
\end{align}
and
\begin{align}\label{D^2norm}
\|\nabla^2 B(\varphi)\|^2_{H.S.}&=\sum_{i,j}\| D_{\hat e_i}D_{\hat e_j}B(\varphi)\|^2_{1-\alpha,s} \nonumber \\
&= \sum_{i,j} \frac{(i+j)^{2(1-\alpha)}(1+a^2(i+j)^2)^{(1-\alpha)s}}{i^4j^4(1+a^2i^2)^{2s}(1+a^2j^2)^{2s}} (\alpha_{j,i+j}+ \alpha_{i,i+j})^2.
\end{align}

\begin{prop}\label{reg_B_derivative}
For all $\alpha>2$ and $p\geq 1$, $\nabla B\in L^p_{\mu_\gamma}(H^{1-\alpha,s}; H.S.(H^{2,s},H^{1-\alpha,s}))$ and $\nabla^2 B\in L^p_{\mu_\gamma}(H^{1-\alpha,s}; H.S.(H^{2,s}\otimes H^{2,s} ,H^{1-\alpha,s}))$.
\end{prop}
\begin{proof}
\begin{align*}
\Em\|\nabla &B(\varphi)\|^{2p}_{H.S.}= \Em(\|\nabla B(\varphi)\|^{2}_{H.S.})^p\\
& =  \Em\left[ \sum_{j,k}\frac{k^{2(1-\alpha)}(1+a^2k^2)^{s(1-\alpha)}}{j^4(1+a^2j^2)^{2s}}
(\alpha_{j,k}+ \alpha_{k-j,k})^2|\w_{k-j}|^2 \right]^p \\
&= \Em \sum_{\substack{k_1, ..., k_p \\ j_1, ...,j_p}} \prod_{i=1}^p \frac{k_i^{2(1-\alpha)}(1+a^2k_i^2)^{s(1-\alpha)}}{j_i^4(1+a^2j_i^2)^{2s}}
(\alpha_{j_i,k_i}+ \alpha_{k_i-j_i,k_i})^2|\w_{k_i-j_i}|^2 \\
&\leq \sum_{\substack{k_1, ..., k_p \\ j_1, ...,j_p}} \prod_{i=1}^p \frac{k_i^{2(1-\alpha)}(1+a^2k_i^2)^{s(1-\alpha)}}{j_i^4(1+a^2j_i^2)^{2s}}
(\alpha_{j_i,k_i}+ \alpha_{k_i-j_i,k_i})^2( \Em |\w_{k_i-j_i}|^{2p})^{1/p}\\
&=\left[ \sum_{j,k}\frac{k^{2(1-\alpha)}(1+a^2k^2)^{s(1-\alpha)}}{j^4(1+a^2j^2)^{2s}}
(\alpha_{j,k}+ \alpha_{k-j,k})^2 \frac{(2^p p!)^{1/p}}{\gamma (k-j)^4(1+a^2(k-j)^2)^{2s}}\right]^p\\
&=\frac{(2^p p!)}{\gamma^p} C^p<+\infty
\end{align*}
As we have shown in the proof of Proposition \ref{B_regularity}, the series above are convergent for every $\alpha>2$ and $p\geq 1$.
For the second derivative of $B$, the statement  follows straightforward from the fact that \eqref{D^2norm} converges for every $\alpha>2$ and $p\geq 1$.
\end{proof}

\subsection{The vector field is divergence free}
Recall that on an abstract Wiener space $(X,H,\mu_\gamma)$, the divergence of a vector field $\Psi: X \to G$, $\Psi\in L^2_{\mu_\gamma}(X;G)$, where $G$ is a Hilbert space, is defined by 
\begin{equation}\label{divergence*}
\int \delta_{\mu_\gamma}\Psi \cdot f d\mu_\gamma=\int (\Psi ,\nabla f )_G d\mu_\gamma,
\quad \forall f\in \mathcal{D}
\end{equation}
where $\mathcal{D}$ is the space of differentiable functions on $X$ depending on a finite number of coordinates, that is $f(u)=f(u_{\alpha_1},..., u_{\alpha_d})$ where $d=d(n)$ and $(\cdot,\cdot)_G$ is the inner product of $G$.

Following \cite{AC}, where the Euler equation is treated, we show that the averaged-Euler vector field is $\mu_\gamma$-divergence free.

\begin{thm}\label{divergence_free}
For $\alpha>2$ the vector field $B: H^{1-\alpha,s} \to H^{1-\alpha,s}$ defined above is divergence free with respect to the measure $\mu_\gamma$, that is $\delta_{\mu_\gamma} B=0$.
\end{thm}
\begin{proof}
Consider $d\mu_\gamma^n = \prod_{k\in\{\alpha_1, ...,\alpha_d\}} d\mu_{\gamma,k}$ where $d=d(n)$ and denote by $\rho_\gamma^n$ the density of this measure with respect to the Lebesgue measure.
From the definition of divergence of a vector field and the fact that $B^n$ converges to $B$ in $L^2_{\mu_\gamma}(H^{1-\alpha,s}; H^{1-\alpha,s})$ when $n$ goes to infinity, we have, for any $f\in \mathcal{D}$,
\begin{align*}
\int \delta_{\mu_\gamma}&B \cdot f d\mu_\gamma(\varphi)=\int <B ,\nabla f >_{1-\alpha,s} d\mu_\gamma(\varphi)\\
& = \lim_n \int \sum_{k>0} k^{2(1-\alpha) }(1+a^2k^2)^{s(1-\alpha)}B^n_k \overline{(\nabla f)_k }d\mu_\gamma^n(\varphi) =  \lim_n \int <B^n \rho_\gamma^n, \nabla g >_{\C^d}dz
\end{align*}
where $g\in \mathcal{D}$ is defined by $g_k=k^{2(1-\alpha)} (1+a^2k^2)^{s(1-\alpha)}f_k$ for all $k\in\{\alpha_1, \ldots, \alpha_{d(n)}\}$. Therefore, for all $g\in \mathcal{D}$, we have 
$$\lim_n \int  \div (B^n \rho_\gamma^n)  g dz = \lim_n \int \left[  \div B^n + <B^n, \frac{\nabla \rho_\gamma^n}{\rho_\gamma^n}>_{\C^d}\right] g d\mu_\gamma^n(\varphi).$$

In particular,
$$
\delta_{\mu_\gamma^n} B^n= \div B^n + <B^n, \frac{\nabla \rho_\gamma^n}{\rho_\gamma^n}>_{\C^d}=0,
$$
since on one hand, from the definition of $B^n$ (equations \eqref{B_n} and \eqref{B_n^k})
$$
\div B^n(\varphi)= \sum_{k\in \{\alpha_1, ..., \alpha_{d(n)}\}} D_{e_k} B^n_k(\varphi)=0
$$
and on the other hand, 
\begin{align*}
<B^n, \frac{\nabla \rho_\gamma^n}{\rho_\gamma^n}>_{\C^d}
= -\gamma<B^n(\varphi), \varphi>_{2,s}
=0
\end{align*}
where the last equality holds by the conservation of the enstrophy.
Therefore $\delta_{\mu_\gamma^n} B^n=0$ for all $n\in \N$ and  $\delta_{\mu_\gamma} B=0$.
\end{proof}

Using the fact that $B$ belongs to $L^2_{\mu_\gamma}$  and has divergence zero (with respect to the measure $\mu_\gamma$), it is possible to construct a flow associated to $B$ for which  $\mu_\gamma$ is an invariant measure.

\begin{lem}\label{finite_dimensional_result}
There exists a unique solution of $\frac{dU^n(t,\varphi^n)}{dt}=B^n(U^n(t,\varphi^n)),$ $U^n(0,\varphi^n)=\varphi^n(0)$ which is defined for all times.
\end{lem}
\begin{proof}
For each $k>0$, $B^n_k(\varphi^n)$ is a finite sum of quadratic terms, 
$$B^n_k(\varphi^n)=\sum_{h^2\leq n}\alpha_{h,k}\varphi^n_h \varphi^n_{k-h}.$$ Then existence of a unique solution follows from classical results on ordinary differential equations, while the conservation of the energy ensure that the solution is globally defined in time.
\end{proof}

Denote by $U^n(t,\varphi^n)$ the flow associated to $B^n$, that is $\varphi^n(0)\mapsto \varphi^n(t)$, and define the flow on $H^{1-\alpha,s}$ by
$$
U^n(t,\varphi)=U^n(t,\varphi^n) + \Pi_n^\perp \varphi,
$$
where $\Pi_n \varphi=\varphi^n$ stands for the orthogonal projection of $u$ on the subspace spanned by $\{e_{\alpha_1}, ..., e_{\alpha_{d(n)}}\}$, then we have
$$
\frac{dU^n(t,\varphi)}{dt}=B^n(U^n(t,\varphi)), \, U^n(0,\varphi)=\varphi(0),
$$
in particular, $U^n(\cdot,\varphi)=\sum_k U^n_k(\cdot,\varphi)e_k$ where $U^n_k(\cdot,\varphi)\in C(\R; \C)$ for all $k>0$.

\begin{thm}\label{exists}
There exists a flow $U(t,\w)$ defined on a probability space $(\Omega, \mathcal{F}, P_\gamma)$ with values in $H^{1-\alpha,s}$, $\alpha>2$, $U(\cdot,\w)\in C(\R; H^{1-\alpha,s})$, $\w\in \Omega$ such that
\begin{enumerate}
\item\label{uno}
\begin{equation}
U_k(t,\w)=U_k(0,\w)+\int_0^t B_k(U(s,\w)) ds, \quad P_\gamma-a.e.\,\w, \,\, \forall t\in\R,
\end{equation}
\item\label{due}
and such that the measure $\mu_\gamma$ is invariant for the flow, in the sense that:
\begin{equation}
\int f(U(t,\w))dP_\gamma(\w)=\int f d\mu_\gamma, \quad \forall t \in \R, \forall f\in \mathcal{D}.
\end{equation}
\end{enumerate}
\end{thm}
\begin{proof}
The construction of such a flow can be found in \cite{AC} in the case of the two-dimensional Euler system. The same arguments apply in the case of the two-dimensional averaged-Euler equations. For $t\in\R^+$ consider $U^n_k(t, \varphi)$ as a stochastic process with law on $C(\R^+; \C)$ defined by
$$
\eta^n_k(\Gamma)=\mu_\gamma\{\varphi : U^n_k(\cdot, \varphi)\in \Gamma\}, \qquad \Gamma \subset C(\R^+; \C).
$$
Consider the sup-norm on $C(\R^+; \C)$ and the weak topology on the space of measures over $C(\R^+; \C)$. We have that:
\begin{enumerate}[label=\arabic*.]
\item \label{one}
$$\eta^n_k(|y(0)|>R)\leq \frac{1}{R^2}\Em(|\w_k|^2)=\frac{2}{\gamma R^2k^4(1+a^2k^2)^{2s}}\to 0 \quad \mbox{ when } R\to +\infty$$ 
\item \label{two}
for all $\rho>0$ and $T>0$ 
\begin{align*}
\eta^n_k\left(\sup_{\substack{0\leq t\leq t'\leq T \\ |t'-t|\leq \delta}} |y(t)-y(t')|>\rho\right)&\leq \frac{1}{\rho^2}
\Em\left( \sup_{t,t'}|U^n_k(t,\varphi)- U^n_k(t',\varphi)|^2\right)\\
&\leq \frac{\delta}{\rho^2}\Em \int_0^T |B^n_k(U^n(s,\varphi))|^2ds\\
&\leq \frac{\delta T}{\rho^2}\Em  |B^n_k|^2 \leq \frac{\delta T C}{\rho^2}\to 0 \mbox{ when } \delta\to 0,
\end{align*}
\end{enumerate}
in the last inequalities we used respectively that $U^n(t,\varphi)$ is a flow for $B^n$ and that $B^n$ has null divergence with respect to $\mu_\gamma$ for all $n$. By \ref{one} and \ref{two} we are under the assumptions of Prohorov's criterium; then there exists a subsequence of $\eta_k^n$ (again denoted by $\eta^n_k$) that converges weakly to $\eta_k$. Remark that we can choose an arbitrary subsequence since $k$ belongs to $\Z^2$ that is countable. Hence, by Skorohod's theorem, there exists a probability space $(\Omega, \mathcal{F}, P_\gamma)$ and family of processes $U'^{n}_k(t,\w)$, $U_k(t,\w)$, $\w\in \Omega$, having laws respectively $\eta^n_k$ and $\eta_k$ on $C(\R^+; \C)$. Furthermore, $U'^{n}_k(\cdot,\w)\to U_k(\cdot,\w)$, $P_\gamma$-a.e. $\w$. Repeat for $t\in\R^+\mapsto U^n_k(-t,\varphi)$ to get the negative values of $t$. We now prove \ref{due}, take $f\in \mathcal{D}$,
\begin{align*}
\int f(U^n(t,\varphi))d\mu_\gamma &= \int d\mu_\gamma^{n,\perp}  \int f(U^n(t,\varphi^n))d\mu_\gamma^{n}\\
&= \int d\mu_\gamma^{n,\perp}  \int f d\mu_\gamma^{n}=\int f d\mu_\gamma, \qquad \forall t>0
\end{align*}
where $d\mu_\gamma^{n}=\prod_{k\in\{\alpha_1, ..., \alpha_{d(n)}\}} d\mu_{\gamma,k}$ and $d\mu_\gamma^{n,\perp}=\prod_{k\notin\{\alpha_1, ..., \alpha_{d(n)}\}} d\mu_{\gamma,k}$. On the other hand denoting by $\eta^n$ the law of $U^n(\cdot, \varphi)$, we also have
\begin{align*}
\int f d\mu_\gamma& =\int f(U^n(t,\varphi)) d\mu_\gamma= \int f (y(t))d\eta^n \\
&= \int f(U'^{n}(t,\w)) dP_\gamma\to \int f(U(t,\w)) dP_\gamma.
\end{align*}
Remark that $U(t,\w)$ takes values in $H^{1-\alpha,s}$; in fact $P_\gamma$-a.e. $\w\in \Omega$ we have
$$
\int \|U(t,\w)\|_{1-\alpha,s}^2dP_\gamma= \int \|\varphi\|_{1-\alpha,s}^2d\mu_\gamma <+\infty.
$$
Finally we prove \ref{uno}; we have
\begin{align*}
\int \Bigl| & \int_0^t [B^n_k(U'^{n}(s,\w))  - B_k(U(s,\w))]ds \Bigr| dP_\gamma \\
&\leq  \int  \int_0^t |B^n_k(U'^{n}(s,\w))  - B_k(U'^{n}(s,\w))|ds dP_\gamma \\
&  + \int  \int_0^t |B_k(U'^{n}(s,\w))  - B_k(U(s,\w))|ds dP_\gamma.
\end{align*}
The first integral converges through zero by the identification in law of $U'^n(t,\w)$ and $U^n(t,\varphi)$, by the invariance of $\mu_\gamma$ under the flow $U^n(t,\varphi)$ and the fact that $B^n_k\to B_k$ in $L^2_{\mu_\gamma}$.
The second integral converges towards zero by the dominated convergence theorem. Indeed $\{B_k(U'^{n}(s,\w))\}_{n\in \N^*}$ is uniformly integrable on $[0,t]\times \Omega$, 
$$
\int  \int_0^t |B_k(U'^{n}(s,\w))| ds dP_\gamma= \int_0^t  \int  |B_k^n(\varphi)|  d\mu_\gamma ds \leq   \int_0^t \int  |B_k(U(s,\w))| dP_\gamma ds  \leq Ct,
$$
and $B_k(U'^{n}_k(s,\w))\to B_k(U_k(s,\w))$ $P_\gamma$-a.e. $\w$ for all $s\in[0,t]$ when $n$ goes to infinity. The latter statement follows from the fact that
$U'^{n}_k(\cdot,\w)\to U_k(\cdot,\w)$, $P_\gamma$-a.e. $\w$ and that $B_k^n$ converges uniformly to $B_k$ (see Corollary \ref{convergence***}) where
$$
B_k(U'^{n}(\cdot,\w))=\sum_{h^2\leq n} \alpha_{h,k} U'^{n}_h(\cdot,\w)U'^{n}_{k-h}(\cdot,\w).
$$
\end{proof}


\begin{rem}
At this point we could ask ourselves about the possibility of considering the (Gibbs) measure associated to the energy, formally $$\nu_\gamma\simeq \frac{1}{Z} e^{-\frac{\gamma E}{2}}\times ``Lebesgue~ measure",$$ where $Z$ denotes a suitable normalizing constant, instead of  the measure associated to the enstrophy $\mu_\gamma$ in \eqref{enstrophy_measure}.
We can observe that the vector field $B$ is not in $L^2$ with respect to $\nu_\gamma$.
\end{rem}

\section{A surface measure}

The energy of the averaged-Euler system belongs to the space $L^2_{\mu_\gamma}$. Therefore, as previously done in \cite{C} for the Euler system 
(here a ``renormalized'' energy must be taken into account, because the energy is not square integrable with respect to the invariant measure), we consider the ``surface'' measure defined on the level sets of $\E$, namely the conditional measure $\mu_\gamma(dx | \E=r)$ for $r>0$. We want to take advantage of the fact that the energy $\E$ is also a conserved quantity of the motion in order to construct a flow for the averaged-Euler vector field with values on the level sets of $\E$. 

\begin{rem}
It is not possible to construct a flow on the level sets of $\E$ using the invariant measure; in fact $\mu_\gamma\{\varphi \, | \E(\varphi)=r\}=0$.
\end{rem}

We consider suitable Sobolev spaces on $(H^{1-\alpha,s},H^{2,s}, \mu_\gamma)$: the space $W_1^p$ of the maps $f:H^{1-\alpha,s}\to \R$ that belong to $L^p_{\mu_\gamma}(H^{1-\alpha,s}; \R)$ such that $\nabla f: H^{1-\alpha,s}\to H^{2,s}$, defined as  $D_h f(x)=<\nabla f(x), h>_{2,s}$ for all $h\in H^{2,s}$ satisfy  $\nabla f\in L^p_{\mu_\gamma}(H^{1-\alpha,s}; H^{2,s})$. More generally the space $W_r^p$, for every integer $r>1$, is the space of functions $f\in W_{r-1}^p$ such that $D_h f(x)\in W_{r-1}^p$ for all $h\in H^{2,s}$.

\begin{prop}\label{reg_ren_energy}
The energy 
$\E$ belongs to Sobolev spaces of all orders, that is \\$\E\in  \W8:=\bigcap_{p,r}W_r^p$.
\end{prop}
\begin{proof}
First, we want to show that
$$
(\Em |\E(\varphi)|^{2m})^{1/m}<+\infty \quad \forall m=2^{p-1} \mbox{ and } p\geq 2.
$$
We have
$$
(\Em |\E(\varphi)|^{2m})^{1/m}\leq \sum_k \left[\Em \left(k^2(1+a^2k^2)^s|\w_k|^2 \right)^{2m}\right]^{1/m} 
$$
and
$$
\Em \left(k^2(1+a^2k^2)^s|\w_k|^2 \right)^{2m}\leq c(p,\gamma)\frac{1}{k^{4m}(1+a^2k^2)^{2ms}}; 
$$
thus
$$
(\Em |\E(\varphi)|^{2m})^{1/m}\leq c(p,\gamma)\sum_k \frac{1}{k^{4}(1+a^2k^2)^{2s}}.
$$
Now consider the linear functional $\nabla \E (\varphi): H^{2,s} \to \R$,
$$
\nabla \E (\varphi)(e_k)=\lim_{\varepsilon\to 0}\frac{1}{\varepsilon}\left(\E(\varphi +\varepsilon e_k)-\E(\varphi)\right)=2k^2(1+a^2k^2)^s|\w_k|,
$$
and take $\hat e_k = \frac{e_k}{k^2(1+a^2k^2)^s}$ for all $k>0$, orthonormal basis of $H^{2,s}$; then
$\nabla \E (\varphi)(\hat e_k)=2|\w_k|$ and 
$$
(\Em\|\nabla \E(\varphi)\|_{2,s}^{2m})^{1/m}\leq 4\sum_k (\Em|\w_k|^{2m})^{1/m}\leq c(p,\gamma)\sum_k \frac{1}{k^4(1+a^2k^2)^{2s}}<+\infty.
$$
Finally observe that 
$$
(\Em\|\nabla^2 \E(\varphi)\|_{H^{2,s}\otimes H^{2,s}}^{2m})^{1/m}=\sum_k \frac{1}{k^4(1+a^2k^2)^{2s}}<+\infty.
$$
\end{proof}

Next proposition is proved in \cite{C}, following \cite{M}, in the case of the Euler system.

\begin{prop}\label{max_rank_energy}
$\E$ is of maximal rank, that is $\|\nabla \E\|^{-1}_{2,s}\in\W8$.
\end{prop}
\begin{proof}
We want to show that $\Em \|\nabla \E(\varphi)\|_{2,s}^{-2p}<+\infty$ for all $p$. By Chebycheff inequality, for all $t>0$
$$
-e^{-\frac{t}{\varepsilon}}\Em \left( e^{-t\|\nabla \E(\varphi)\|_{2,s}^{-2}}\right)\leq\mu_\gamma \left\{ \|\nabla \E(\varphi)\|_{2,s}^{2}\leq \varepsilon \right\}\leq 
e^{\frac{t}{\varepsilon}}\Em \left( e^{-t\|\nabla \E(\varphi)\|_{2,s}^{2}}\right).
$$
In particular
$$
\mu_\gamma \left\{ \|\nabla \E(\varphi)\|_{2,s}^{2}\leq \varepsilon \right\} \geq -e^{-\frac{1}{\varepsilon}}\sum_{p\geq 0}\frac{(-1)^p}{p!}\Em(\|\nabla \E(\varphi)\|_{2,s}^{-2p}),
$$
meaning that $\Em \|\nabla \E(\varphi)\|_{2,s}^{-2p}$ are finite for all $p$ whenever $\mu_\gamma \left\{ \|\nabla \E(\varphi)\|_{2,s}^{2}\leq \varepsilon \right\}$ is finite. We have
\begin{align*}
\mu_\gamma \left\{ \|\nabla \E(\varphi)\|_{2,s}^{2} \leq \varepsilon \right\}&\leq 
e^{\frac{t}{\varepsilon}}\Em \left( e^{-t\|\nabla \E(\varphi)\|_{2,s}^{2}}\right)\\
& = e^{\frac{t}{\varepsilon}}\prod_k \left( \frac{1}{1+ \frac{8t}{\gamma k^4(1+a^2k^2)^{2s}}} \right)\\
&\leq e^{\frac{t}{\varepsilon}}\prod_{\{k \,:\, \gamma k^4(1+a^2k^2)^{2s}<\frac{8}{t}\}} \left( \frac{1}{1+t^2} \right)\\
&\leq \inf_t e^{\frac{t}{\varepsilon}}\left( \frac{1}{1+t^2} \right)^c <+\infty
\end{align*}
where $c=\#\{k\,:\, \gamma k^4(1+a^2k^2)^{2s}<\frac{8}{t}\}$.
\end{proof}

For $g\in\W8$, we shall denote by $\rho(r)=\frac{d (\E * \mu_\gamma)}{dr}$ and by $\rho_g(r)=\frac{d (\E * g\mu_\gamma)}{dr}$ respectively the $C^{\infty}$ densities of $d (\E * \mu_\gamma)$ and $d (\E * g\mu_\gamma)$ with respect to the Lebesgue measure, see \cite{Ma,AM}. As proved in \cite{AM}, Propositions \ref{reg_ren_energy} and \ref{max_rank_energy} ensure the existence of a conditional measure of $\mu_\gamma$ knowing that $\E=r$ for $r>0$. 

\begin{thm}\label{surf_measure}
Let $r>0$ be such that $\rho(r)>0$; then there exists a Borel probability measure defined on $H^{1-\alpha,s}$, $\nu^r_\gamma$, with support on $V_r=\{\varphi\, | \E(\varphi)=r\}$ and such that
$$
\int g^*(\varphi)d\nu^r_\gamma=\frac{\rho_g(r)}{\rho(r)},
$$
for any $g^*$redefinition of $g$.
\end{thm}
\begin{proof}
See \cite{AM}.
\end{proof}

\begin{rem}
Recall that, given a measurable function $\Phi$ with values in $\R^n$, we call a $(p,r)$-redefinition of $\Phi$ a function $\Phi^*$
 such that $\Phi=\Phi^*$ a.s. and $\Phi^*$ is $(p,r)$-continuous (that is, if $\forall \varepsilon >0$ it is possible to find an open set $O_\varepsilon$ such that $c_{p,r}(O_\varepsilon)<\varepsilon$ and the restriction of $\Phi^*$ to $O_\varepsilon^c$ is continuous). The capacity of the open set $O$ is given by $c_{p,r}(O)=\inf \{ \|u\|_{W_{2r}^p}; u\geq 0, u(x)\geq 1 , \mu-\mbox{a.e. on } O \}$; $O$ is said to be slim if $c_{p,r}(O)=0$, for all $p,r\in \mathbb{N}$.

For all $\Phi\in \W8$ there exists a redefinition $\Phi^*$ and a sequence of open sets $\{O_n\}_{n\in \N}$ associated to this redefinition such that: $\bigcap_n O_n$ is slim, $\Phi^*$ is continuous on $\left(\bigcap_n O_n\right)^c$ and $\Phi^*$ and $\nabla^r \Phi^*$ are continuous on $O_n^c$ for all $n,r \in \N$. 

The proof of Theorem \ref{surf_measure} is essentially based on the following considerations. For a fixed $\Phi\in \W8$ of maximal rank and non-degenerate and for $g\in\W8$ we consider the map 
$$
<\delta\Phi, g>: \xi \mapsto <\delta_\xi\Phi, g>:=\rho_g(\xi)/\rho(\xi);
$$
this map  belongs to $C^{\infty}(O;\R)$ where $O=\{\xi \in \mbox{supp}(\Phi*\mu)\subset\R^n : \rho(\xi)>0 \}$. 
In particular the map 
$$
g\mapsto <\delta\Phi, g>
$$
is a continuous linear functional from $\W8$ to the space of functions $C^{\infty}$ on $O$. If $S(O)$ is the Schwartz space of $O$ and $W'$ the dual of $\W8$ ($W'$ was accurately defined by Watanabe, see \cite{Ma}) we can consider the dual map 
$$
\delta_*\Phi: S(O) \to W'
$$ 
that associates linear functionals on $\W8$ to distributions over $\R^n$ and such that
\begin{equation}\label{duality}
< <\delta_*\Phi, v>,g>=<v,<\delta\Phi, g> >
\end{equation}
for every $v\in S(O)$ and $g\in\W8$. For further details see \cite{Ma, AM}.
\end{rem}

We compute the second order moments of $\nu_\gamma^r$. From \cite{Ma} we know that $\rho_{\w_k \bar\w_{k'}}\in S(\R)$ since $\w_k\bar \w_{k'}\in \W8$; then we have 
$$
\hat \rho_{\w_k\bar \w_{k'}}(r) = \int_{-\infty}^{+\infty} e^{i r \xi}\rho_{\w_k\bar \w_{k'}}(\xi) d\xi=  \int_{H^{1-\alpha,s}} e^{i r \E(\varphi)} \w_k\bar \w_{k'} d\mu_\gamma.
$$
If $k\neq k'$, 
\begin{align*}
\hat \rho_{\w_k\bar \w_{k'}}(r) = \prod_{j\neq k, k'}&\int_\C e^{\frac{ir}{2}j^2(1+a^2j^2)^s|\w_j|^2}d\mu_{\gamma,j}\int_\C e^{\frac{ir}{2}k^2(1+a^2k^2)^s|\w_k|^2 }\w_k d\mu_{\gamma,k}\\
&\int_\C e^{\frac{ir}{2}k'^2(1+a^2k'^2)^s|\w_k'|^2}\w_{k'} d\mu_{\gamma,k'}=0,
\end{align*}
if $k=k'$,
\begin{align*}
\hat \rho_{\w_k\bar \w_{k'}}(r)& = \prod_{j\neq k}\int_\C e^{\frac{ir}{2}j^2(1+a^2j^2)^s|\w_j|^2 }d\mu_{\gamma,j}\int_\C e^{\frac{ir}{2}k^2(1+a^2k^2)^s|\w_k|^2 }|\w_k |^2d\mu_{\gamma,k},
\end{align*}
where
\begin{align*}
\int_\C & e^{\frac{ir}{2}k^2(1+a^2k^2)^s|\w_k|^2}|\w_k |^2d\mu_{\gamma,k}\\
&= \frac{\gamma k^4 (1+ a^2k^2)^{2s}}{2\pi} \int_\C e^{\frac{ir}{2}k^2(1+a^2k^2)^s|\w_k|^2-\frac{\gamma}{2} k^4(1+ a^2k^2)^{2s} |\w_k|^2}\w_k  \bar\w_k dz\\
&=\frac{1}{\gamma k^4 (1+ a^2k^2)^{2s} - ir k^2(1+a^2k^2)^s}\int_\C e^{\frac{ir}{2}k^2(1+a^2k^2)^s|\w_k|^2 }d\mu_{\gamma,k}
\end{align*}
after complex by parts integration.
Then
$$
\hat \rho_{\w_k\bar \w_{k'}}(r)= \frac{1}{\gamma k^4 (1+ a^2k^2)^{2s} - ir k^2(1+a^2k^2)^s}\int_{H^{1-\alpha,s}} e^{ir\E(\varphi)}d\mu_\gamma.
$$
Hence
\begin{align*}
\rho_{\w_k\bar \w_{k'}}(\xi)&=\frac{1}{k^2(1+a^2k^2)^s}\int_{-\infty}^{+\infty} e^{-ir\xi}\frac{1}{\gamma k^2(1+a^2k^2)^s-ir}\int_{H^{1-\alpha,s}} e^{ir\E(\varphi)} d\mu_\gamma dr\\
&=\frac{1}{k^2(1+a^2k^2)^s}\int_{-\infty}^{+\infty} e^{-ir\xi}\frac{\hat \rho(\xi)}{\gamma k^2(1+a^2k^2)^s-ir}dr\\
\end{align*}
where
\begin{align*}
\int_{-\infty}^{+\infty} e^{-ir\xi}\frac{\hat \rho(\xi)}{\gamma k^2(1+a^2k^2)^s-ir}dr&=\left( \hat \rho(\xi)\frac{1}{\gamma k^2(1+a^2k^2)^s-ir} \right)^\vee\\
&=\rho(\xi)\ast \left(2\pi e^{-\gamma k^2(1+a^2k^2)^s y}\right),
\end{align*}
then
$$
\rho_{\w_k\bar \w_{k'}}(\xi)=\frac{\pi}{k^2(1+a^2k^2)^s}\int_0^{+\infty}\rho(\xi+y) e^{-\gamma k^2(1+a^2k^2)^s y} dy.
$$
We conclude that  $\En(\w_k\bar \w_{k'})=\frac{\rho_{\w_k\bar \w_{k'}}(r)}{\rho(r)}=0$ if $k\neq k'$ and 
$\En(\w_k\bar \w_{k'})=\frac{\pi}{k^2(1+a^2k^2)^s\rho(r)}\int_0^{+\infty}\rho(r+y) e^{-\gamma k^2(1+a^2k^2)^s y}dy$
if $k=k'$.

\section{The invariant flow}

\subsection{Existence}

Similar to \cite{C}, we show that the vector field $B$ is divergence free with respect to the surface measure $\nu^r_\gamma$. This will be fundamental for proving the existence of a flow on the level sets of the energy.
\begin{thm}\label{surf_meas_inv}
$$\int <B^n, \nabla f >^{*}_{2,s}d\nu^r_\gamma=0, \qquad \forall f\in \mathcal{D}$$
for any $<B^n, \nabla f >^{*}_{2,s}$ redefinition of $<B^n, \nabla f >_{2,s}$.
\end{thm}
\begin{proof}
Let $f\in \mathcal{D}$ and $v\in C^{\infty}_0(\R)$ be arbitrary functions. We have,
\begin{align*}
\int_\R v(r)\rho(r)&\int_{V_r}  <B^n, \nabla f >^{*}_{2,s} d\nu^r_\gamma dr= \int_\R v(r) d (E * <B^n, \nabla f >_{2,s }\mu_\gamma)\\
&= \int_{H^{1-\alpha,s}}  <v(\E(\varphi))B^n, \nabla f >_{2,s }d\mu_\gamma \\
&= \int_{H^{1-\alpha,s}}
\delta_{\mu_\gamma} \left( v(\E(\varphi))B^n \right) f d\mu_\gamma \\
&=  \int_{H^{1-\alpha,s}} \left[  v(\E(\varphi))\delta_{\mu_\gamma} B^n - v'(\E(\varphi))<B^n, \nabla \E(\varphi) >_{2,s } \right]f d\mu_\gamma \\
&=0,
\end{align*}
because, as we saw in Theorem \ref{divergence_free}, $\delta_{\mu_\gamma} B^n=0$ and
$$
<B^n, \nabla \E(\varphi) >_{2,s }=2<B( \varphi^n),  \varphi^n>_{1,s}
=0
$$
since the energy is conserved.
\end{proof}

In order to prove existence of a global averaged-Euler flow defined $\nu^r_\gamma$ almost everywhere  and taking values on the level sets of the energy $\E$, recall the finite dimensional result of Lemma \ref{finite_dimensional_result} and that we denoted the flow associated to $B^n$ on $H^{1-\alpha,s}$ by $U^n(t,\varphi)=U^n(t,\varphi^n) + \Pi_n^\perp \varphi$.

\begin{thm}\label{flow_level_sets}
Let $\alpha>2$. For all $r>0$ such that $\rho(r)>0$, there exists a flow $U'(\cdot, \w)$ defined on a probability space $(\Omega, \mathcal{F}, P^r_\gamma)$ with values in $V_r$, $U'(\cdot, \w) \in C(\R; V_r)$, $\w \in \Omega$ such that:
\begin{enumerate}
\item \label{1}
for any $B^*$ redefinition of $B$,
$$
U'_k(t,\w)=U'_k(0,\w) + \int_0^t B_k^*(U'(s,\w))ds, \quad P^r_\gamma-a.e. \,\w, \,\, \forall t\in\R,
$$
\item \label{2} $\nu^r_\gamma$ is invariant for the flow, in the sense that:
$$
\int f(U'(t,\w))d P^r_\gamma(\w)=\int f(\varphi)d\nu^r_\gamma(\varphi), \qquad \forall t \in \R,\,\, \forall f\in \mathcal{D}.
$$
\end{enumerate}
\end{thm}

Before the proof of the theorem we give a complementary Lemma.

\begin{lem}\label{conv_in_L^2_nu}
The approximated averaged-Euler vector field $B^n$ converges to $B$ in \\ $L^2_{\nu^r_\gamma}(H^{1-\alpha,s};H^{1-\alpha,s})$ as $n\to \infty$.
\end{lem}
\begin{proof}
From the results on the regularity of $B$, Subsection \ref{reg}, we get $B\in \W8(H^{1-\alpha,s})$ ($\W8(H^{1-\alpha,s})$ denotes the space of functions $\W8$ with values in $H^{1-\alpha,s}$) and therefore
\begin{equation}\label{reg_B*}
\int_{V_r} (\|B(\varphi)\|^2_{1-\alpha,s})^* d\nu^r_\gamma <\infty.
\end{equation}
Also from the results of Subsection \ref{reg} it follows that $B^n$ is a Cauchy sequence in $W^p_r(H^{1-\alpha,s})$ for all $r,p$ and by definition of $\nu^r_\gamma$ we have 
$$
\rho(r)\int_{V_r} (\|B^n-B\|_{1-\alpha,s}^2)^* d\nu^r_\gamma =\rho_{\|B^n-B\|_{1-\alpha,s}^2}(r).
$$
Hence if we show that $\rho_{\|B^n-B\|_{1-\alpha,s}^2}(r)$ converges to zero as $n$ tends to infinity, we get the lemma. 
From \cite{AM,Ma} we know that $\rho_{\|B^n-B\|_{1-\alpha,s}^2}(r)\leq 
\left\| \|B^n-B\|_{1-\alpha,s}^2 \right\|_{W^p_r}$ while 
$$
\left\| \|B^n-B\|_{1-\alpha,s}^2 \right\|_{W^p_r}\leq C \|B^n-B\|_{W^{2p}_r(H^{1-\alpha,s})}^2.
$$
In fact we have  
$$
|D_{\hat e_j} \|B^n-B\|_{1-\alpha,s}^2|\leq 2 \|D_{\hat e_j}(B^n-B)\|_{1-\alpha,s}\|B^n-B\|_{1-\alpha,s}
$$
and thus
$$
\|\nabla \|B^n-B\|_{1-\alpha,s}^2\|_{2,s} \leq 2 \|B^n-B\|_{1-\alpha,s} \|\nabla (B^n-B)\|_{H.S.(H^{2,s}; H^{1-\alpha,s})}.
$$
A similar argument holds for the higher order derivatives.
\end{proof}

In particular, from Lemma \ref{conv_in_L^2_nu}, there exists a constant $C_2$ such that 
$$
\sup_n\int_{V_r}(\|B^n(\varphi)\|^2_{1-\alpha,s})^* d\nu^r_\gamma \leq C_2, \quad \forall \alpha>2.
$$

We finally prove Theorem \ref{flow_level_sets}

\begin{proof}[Proof of Theorem \ref{flow_level_sets}]
Let $t\in\R^+$ and consider $U^n_k$ as a stochastic process with laws on the space $C(\R^+; \C)$ endowed with the sup-norm:
$$
\eta_k^n (\Gamma)=\nu^r_\gamma(\{ \varphi : U^n_k (\cdot, \varphi)\in \Gamma\}), \quad \Gamma\subset C(\R^+; \C).
$$
We consider the weak topology on the space of measures on $C(\R^+; \C)$ . We have 
\begin{enumerate}[label=\arabic*.]
\item $$\eta_k^n(|y(0)|>R)\leq \frac{1}{R^2} \En|\w_k|^2\leq \frac{C_3}{R^2} \to 0 \mbox{  when  }R\to \infty$$
\item  for all $L>0 \mbox{ and } T>0$,
\begin{align*}
\eta_k^n\left(\sup_{\substack{0\leq t\leq t'\leq T \\ t'-t>\delta}} |y(t')-y(t)|>L \right)& \leq \frac{1}{L^2} \En \left( \sup_{t',t}|U^n_k(t',\varphi)-U^n_k(t,\varphi)|^2\right)\\
& \leq \frac{\delta}{L^2}\En \left( \int_0^T |B_k^n(U^n(s,\varphi))|^2 ds \right) \\
& \leq  \frac{T \delta}{L^2} \En |B^n_k|^2 \\
& \leq \frac{T \delta C_2}{L^2} \to 0 \mbox{  when  }\delta\to 0, 
\end{align*}
\end{enumerate}
where in the last inequalities we used respectively that $B^n$ is $\nu^r_\gamma$-invariant for all $n$ and that \\ $\sup_n\En \|B^n(\varphi)\|^2_{1-\alpha,s} \leq C_2$. Hence, by Prohorov's criterium (actually a combined version of Prohorov's criterium and Ascoli-Arzel\`{a} theorem) we can state that there exists a subsequence of $\eta_k^n$ (for simplicity it will also be denoted by $\eta_k^n$), that converges to some probability measure $\eta_k$. We denote by $U_k$ the stochastic process with law $\eta_k$. By Skorohod's theorem, there exists a probability space $(\Omega, \mathcal{F}, P^r_\gamma)$ and a family of processes $U'^{n}(t,\w)$, $U'(t,\w)$ with laws respectively $\eta^n,\eta$. Furthermore $U'^{n}(\cdot,\w)\to U'(\cdot,\w)$ for a.e. $\w\in\Omega$, that is, there exists $A\subset \Omega$ such that $P^r_\gamma (A^c)=0$ and for all $\w\in A$, $U'^{n}(s,\w)\to U'(s,\w)$ in $H^{1-\alpha,s}$ for all $s\in[0,T]$. Repeating the construction for the processes $t\in\R^+\mapsto U^n_k(-t,\varphi)$ we obtain the negative values of $t$. 
We now prove \ref{2}: for all $f$ in $\mathcal{D}$, because  $\delta_{\nu^r_\gamma}B^n=0$ for all $n$ (Theorem \ref{surf_meas_inv}), we have,
$$
\int f(U^n(t,\varphi))d\nu^r_\gamma(\varphi)=\int f(\varphi)d\nu^r_\gamma(\varphi).
$$
On the other hand, by definition of $\eta^n$ and $\eta$,
\begin{align*}
\int f(U^n(t,\varphi))d\nu^r_\gamma(\varphi)& =\int f(y(t))d\eta^n(y(t)) \\
&=\int f(U'^n(t,\w))d P^r_\gamma(\w) \\
& \to \int f(U'(t,\w))d P^r_\gamma(\w) \mbox{ when } n \to \infty.
\end{align*}

To prove \ref{1} it is enough to check that 
$$
\int \left| \int_0^t B_k^n(U'^{n}(s,\w)) - B_k^*(U'(s,\w)) ds \right| dP^r_\gamma
$$
converges to zero when $n$ goes to infinity; we have
\begin{align*}
\int \left| \int_0^t B_k^n(U'^n(s,\w)) - B_k^*(U'(s,\w)) ds \right| dP^r_\gamma & \leq \int_0^T \int \left| B^n_k(U'^n(s,\w))- B_k^*(U'^n(s,\w)) \right| dP^r_\gamma ds \\
&+\int_0^T \int \left|  B_k^*(U'^n(s,\w)) -B_k^*(U'(s,\w)) \right| dP^r_\gamma ds .
\end{align*}

The first integral converges to zero by the $\nu^r_\gamma$-invariance of the flow and because $B^n$ converges to $B$ in $L^2_{\nu^r_\gamma}(H^{1-\alpha,s};H^{1-\alpha,s})$ as we proved in Lemma \ref{conv_in_L^2_nu}. For the second integral consider $D'=[0,T]\times A^c$, clearly $\lambda \times P^r_\gamma (D')=0$ where $\lambda(ds)$ denotes the Lebesgue measure on $\R$.
Also we can find a subset $D\subset H^{1-\alpha,s}$ such that $\nu^r_\gamma(D)=0$ and $B^*$ restricted to $D^c$ is continuous, then define
$$
A_n=\{(s,\w) \in [0,T]\times \Omega : U'^n(s,\w)\in D \}
$$
and 
$$
A_\infty=\{(s,\w) \in [0,T]\times \Omega : U'(s,\w)\in D \}.
$$
We have $\lambda \times P^r_\gamma (A_n)=0$ for all $n$ and $\lambda \times P^r_\gamma (A_\infty)=0$.  Set 
$$
\Delta= A_\infty \cup (\cup_n A_n) \cup D'.
$$
Let $(s,\w)\in \Delta^c$; in particular $U'^n$ and $U'$ take values in $D^c$ (in which $B^*$ is continuos) and $U'^n_k(s,\w)\to U'_k(s,\w)$ in $\C$. We have, 
$$
\left|  B_k^*(U'^n(s,\w)) -B_k^*(U'(s,\w)) \right|  \to 0
$$
and since
$$
\int_0^T \int \left|  B_k^*(U'^n(s,\w)) -B_k^*(U'(s,\w)) \right| dP^r_\gamma ds
$$
is uniformly bounded, by Egoroff's theorem the second integral also converges to zero. 
\end{proof}

\subsection{Return to a neighborhood of its initial state}

The Poincar\'{e} recurrence theorem holds in our particular case. This is used here to prove that the globally defined invariant flow returns infinitely many times in a neighborhood of the initial state. A similar result is proved in \cite{CL} for the one-dimensional Camassa-Holm equation and certain initial profiles for which the solutions exist globally. We recall the Poincar\'{e} recurrence theorem (c.f. \cite{KAC}).

\begin{thm}
Let $P$ be a probability measure defined on a set $\Omega$. If $\{T_t\}_{t\geq 0}$ is a one-parameter family of measure preserving  transformations and $\Omega_0$ is a subset of $\Omega$ with $P(\Omega_0)>0$, then for $P$-almost every $\w\in \Omega_0$ there exist arbitrarily large $t$ such that $T_t\w\in\Omega_0$. 
\end{thm}

\begin{thm}
Let $\alpha>2$ and fix $\varphi_0\in V_r\subset H^{1-\alpha,s}$. If $\varepsilon>0$ is sufficiently small, then for $\nu^r_\gamma$-a.e. $\varphi \in V_r \subset H^{1-\alpha,s}$ such that $\|\varphi - \varphi_0\|_{1-\alpha,s}<\varepsilon$, there exists a sequence $\{t_n\}\uparrow \infty$ such that the corresponding invariant flow starting from $\varphi$, $U_\varphi(t,\w)$, satisfies $\mathbb{E}_{P^r_\gamma}\|U_\varphi(t_n,\w) -\varphi_0\|_{1-\alpha,s}<2\varepsilon$.
\end{thm}
\begin{proof}
The statement follows by applying Poincar\'{e} recurrence theorem to the open set $B(\varphi_0,\varepsilon)=\{\varphi \in V_r \subset  H^{1-\alpha,s} : \|\varphi - \varphi_0\|_{1-\alpha,s}^2<\varepsilon\}$ with $\nu^r_\gamma\{B(\varphi_0,\varepsilon)\}>0$.
\end{proof}

\subsection*{Acknowledgements}
The author wishes to thank Professor Ana Bela Cruzeiro for the introduction to the topic of invariant measures and for the great support during the preparation of this research. The author was funded by the LisMath fellowship PD/BD/52641/2014, FCT, Portugal.


\bibliographystyle{plain}

\vspace{1cm}

E-mail: \tt asymeonides@fc.ul.pt

\end{document}